\documentclass[reqno,11pt]{elsarticle}
\usepackage{amsfonts}
\usepackage{amssymb}
\usepackage{graphicx}
\usepackage{pstricks}
\usepackage{amsmath}
\usepackage{amsmath}
\usepackage{amsxtra}
\usepackage{marginnote}

\usepackage{bm}
\usepackage{ulem}

\usepackage{moreverb}

\definecolor{DarkBlue}{rgb}{0,0.2,0.6}
\definecolor{PinkPurple}{rgb}{0.8,0.3,0.3}
\definecolor{darkgreen}{rgb}{.1,.5,0}
\definecolor{brown}{rgb}{.4,.2,.1}

\newtheorem{theorem}{Theorem}[section]
\newtheorem{lemma}[theorem]{Lemma}
\newproof{proof}{Proof}
\newtheorem{proposition}[theorem]{Proposition}
\newtheorem{corollary}[theorem]{Corollary}

\newtheorem{definition}[theorem]{Definition}

\newtheorem{problem}[theorem]{Problem}
\newtheorem{question}[theorem]{Question}

\newtheorem{example}[theorem]{Example}
\newtheorem{remark}[theorem]{Remark}

\numberwithin{equation}{section}
\def\ker{\operatorname{ker}}

\def\max{\operatorname{max}}

\newcommand\ddfrac[2]{\frac{\displaystyle #1}{\displaystyle #2}}

\numberwithin{equation}{section}
\begin{document}

%\normalem

\begin{frontmatter}

\title{Conditional positive definiteness as a bridge \\ between $k$--hyponormality and $n$--contractivity}
\author{Chafiq Benhida}
\address{UFR de Math\'{e}matiques, Universit\'{e} des Sciences et
Technologies de Lille, F-59655 \newline Villeneuve-d'Ascq Cedex, France}
\ead{chafiq.benhida@univ-lille.fr}

\author{Ra\'ul E. Curto}
\address{Department of Mathematics, University of Iowa, Iowa City, Iowa 52242-1419, USA}
\ead{raul-curto@uiowa.edu}
\ead[url]{http://www.math.uiowa.edu/\symbol{126}rcurto/}

\author{George R. Exner}
\address{Department of Mathematics, Bucknell University, Lewisburg, Pennsylvania 17837, USA}
\ead{exner@bucknell.edu}

\begin{abstract}
For sequences $\alpha \equiv \{\alpha_n\}_{n=0}^{\infty}$ of positive real numbers, called weights, we study the weighted shift operators $W_{\alpha}$ having the property of moment infinite divisibility ($\mathcal{MID}$); that is, for any $p > 0$, the Schur power $W_{\alpha}^p$ is subnormal. \ We first prove that $W_{\alpha}$ is $\mathcal{MID}$ if and only if certain infinite matrices $\log M_{\gamma}(0)$ and $\log M_{\gamma}(1)$ are conditionally positive definite (CPD). \ Here $\gamma$ is the sequence of moments associated with $\alpha$, $M_{\gamma}(0),M_{\gamma}(1)$ are the canonical Hankel matrices whose positive semi-definiteness determines the subnormality of $W_{\alpha}$, and $\log$ is calculated entry-wise (i.e., in the sense of Schur or Hadamard). \ Next, we use conditional positive definiteness to establish a new bridge between $k$--hyponormality and $n$--contractivity, which sheds significant new light on how the two well known staircases from hyponormality to subnormality interact. \ As a consequence, we prove that a contractive weighted shift $W_{\alpha}$ is $\mathcal{MID}$ if and only if for all $p>0$, $M_{\gamma}^p(0)$ and 
$M_{\gamma}^p(1)$ are CPD. \ 
\end{abstract}


\end{frontmatter}

\tableofcontents

\setcounter{tocdepth}{2}

%%%%%%%%%%%%%%%%%%%%%%%%
%%%%%%%%%%%%%%%%%%%%%%%%
%%%%%%%%%%%%%%%%%%%%%%%%
%%%%%%%%%%%%%%%%%%%%%%%%
%%%%%%%%%%%%%%%%%%%%%%%%
%%%%%%%%%%%%%%%%%%%%%%%%

\section{Introduction and Statement of Main Results} \label{Intro}

Let $\mathcal{H}$ denote a separable, complex Hilbert space and $\mathcal{L}(\mathcal{H})$ be the algebra of bounded linear operators on $\mathcal{H}$. \  Recall that an operator $T$ is \textit{subnormal} if it is the restriction to a (closed) invariant subspace of a normal operator, and \textit{hyponormal} if $T^* T \geq T T^*$. \ A unilateral weighted shift $W_{\alpha}$ acting on the classical sequence space $\ell^2(\mathbb{N}_0)$ is called \textit{moment infinitely divisible} (in symbols, $W_{\alpha} \in \mathcal{MID}$) if all Schur powers $W_{\alpha}^p$ are subnormal. \ Thus, the class $\mathcal{MID}$ consists of all subnormal weighted shifts $W_{\alpha}$ with moment sequence $\gamma$ such that $\gamma^p$ is interpolated by a positive Borel probability measure $\mu^{(p)}$, for every $p>0$; $\mu^{(p)}$ is the so-called \textit{Berger measure} of the subnormal weighted shift $W_{\alpha}^p$. \ This is equivalent to the (Schur) infinite divisibility of the two Hankel \textit{moment matrices} $M_{\gamma}(0):=\left(\gamma_{i+j}\right)_{i,j=0}^{\infty}$ and $M_{\gamma}(1):=\left(\gamma_{i+j+1}\right)_{i,j=0}^{\infty}$, where $\gamma$ is the sequence of moments associated with the weight sequence $\alpha$. \ Since all entries in these two matrices are positive, and since the matrices are Hermitian, their infinite divisibility is equivalent to the \textit{conditional positive definiteness} of their (Schur) logarithms $\log M_{\gamma}(0)$ and $\log M_{\gamma}(1)$. 

As a consequence, we can prove that $W_{\alpha}$ is $\mathcal{MID}$ if and only if both $\log M_{\gamma}(0)$ and $\log M_{\gamma}(1)$ are CPD. \ This leads to a new characterization of $W_{\alpha} \in \mathcal{MID}$ in terms of the sequence $\delta_n:=\log \left( \ddfrac{\gamma_n \gamma_{n+2}}{\gamma_{n+1}^2}\right) \quad (n \ge 0)$. \ 

It is well known that subnormality is a much stronger condition than hyponormality. \ For a contractive weighted shift $W_{\alpha}$, the former requires a Berger measure $\mu$ supported in $[0,1]$ that interpolates $\gamma$; the latter requires $\alpha_0 \le \alpha_1 \le \alpha_2 \le \ldots$. \ There are two well-known staircases connecting hyponormality and subnormality, and they correspond to two well-known tests for subnormality. \ The first test involves the Bram-Halmos criterion for subnormality, which requires that $W_{\alpha}$ be $k$--hyponormal for every $k \ge 1$; the second test has to do with the Agler-Embry  approach to subnormality, which requires the $n$--contractivity of $W_{\alpha}$ for every $n \ge 1$. \ 

In this paper we consider the role of conditional positive definiteness in establishing a bridge between the two above mentioned staircases. \ Concretely, as a first step we will see that $k$--hyponormality implies $2k$--contractivity. 

In \cite{EJP}, G.R. Exner, I.B. Jung and S.S. Park proved for general operators that $k$--hyponormality implies $2k$--contractivity. \ Implicit in their work was a significant identity, which we use as a point of departure. \ Given a moment sequence $\gamma$, a positive integer $k$, a nonnegative integer $\ell$, and the $(k+1) \times(k+1)$ (Hankel) compressed moment submatrix $M_{\gamma}(\ell,k):=\left( \gamma_{\ell+i+j} \right)_{i,j=0}^{k}$, we consider the expression
\begin{equation} \label{identity1}
Q_{\gamma}(\ell,k):=\bm{v}^{\ast} M_{\gamma}(\ell,k)\bm{v},
\end{equation}
where $\bm{v}$ is the column vector of length $k+1$ with $i$--th coordinate $(-1)^i \binom{k}{i}$; that is, $\bm{v}:=(1, -\binom{k}{1},
\binom{k}{2},-\binom{k}{3}, \ldots, -\binom{k}{k-1}, 1)$. \ Observe that $\sum_{i=0}^{k} v_i=(1-1)^{k}=0$. \ Experimental calculations using {\it Mathematica} \cite{Wol} easily reveal that (\ref{identity1}) becomes
\begin{eqnarray} \label{identity2}
Q_{\gamma}(\ell,k)&=&\gamma_{\ell}-\binom{2k}{1} \gamma_{\ell+1}+\binom{2k}{2} \gamma_{\ell+2}-\binom{2k}{3} \gamma_{\ell+3} + \ldots \nonumber \\
&&-\binom{2k}{2k-1}\gamma_{\ell+2k-1}+\gamma_{\ell+2k} \nonumber \\
&&= \sum_{i=0}^{2k} (-1)^i \binom{2k}{i}\gamma_{\ell+i}.
\end{eqnarray}
Once the proposed form of $Q_{\gamma}(\ell,k)$ is discovered, a proof of (\ref{identity2}) by induction is straightforward. 

Positivity of the expression on the right-hand side of (\ref{identity2}) for all $\ell$ is exactly what we need for $2k$--contractivity of $W_{\alpha}$. \ On the other hand, (\ref{identity1}) is one of the expressions used in the determination of CPD for the matrix $M_{\gamma}(\ell,k)$. \ Concretely, if $W_{\alpha}$ is $k$--hyponormal, then $M_{\gamma}(\ell,k) \ge 0$ for all $\ell \ge 0$. \ Then $M_{\gamma}(\ell,k)$ is CPD for all $\ell \ge 0$, and therefore $Q_{\gamma}(\ell,k) \ge 0$ for all $\ell \ge 0$. \ It follows that $\sum_{i=0}^{2k} (-1)^i \binom{2k}{i}\gamma_{\ell+i} \ge 0$ for all $\ell \ge 0$, and this means that $W_{\alpha}$ is $2k$--contractive. \ This is a special case of a more general result, involving $(k,2m)$--CPD matrices, which represents a very useful version of CPD, appropriately localized to keep track of the size of the matrix and the initial moment. 

The above calculation is part of a much broader theoretical setting, involving properties of sequences such as complete monotonicity and hypercontractivity, along with their $\log$ analogs, e.g., $\log$ complete monotonicity. \ We discuss aspects of this theory in Section \ref{main}.

We are now ready to state the main results of this paper. \ The proofs will be given in Section \ref{main}. \ For notation and terminology, we refer the reader to Section \ref{prelim}, where we also briefly review some topics from unilateral weighted shifts and matrix theory that will be needed for Section \ref{main}. \ 

\begin{theorem}  \label{thm:equivcondtIDcontraction}
Suppose $W_\alpha$ is a contractive weighted shift. \ Then the following statements are equivalent.
\begin{enumerate}
             \item $W_\alpha$ is moment infinitely divisible ($\mathcal{MID}$).
             \item $\log M_{\gamma}(0)$ and $\log M_{\gamma}(1)$ are CPD.
             \item For every $p>0$, $M_{\gamma}^p(0)$ and $M_{\gamma}^p(1)$ are positive definite.
             \item For every $p>0$, $M_{\gamma}^p(0)$ and $M_{\gamma}^p(1)$ are CPD.
             \item The moment sequence $\gamma$ is log completely monotone.
             \item The weight sequence $\alpha$ is log completely alternating.
\end{enumerate}
\end{theorem}

\begin{proposition}       \label{prop:getAglerfrom1}
Let $W_{\alpha}$ be a unilateral weighted shift, $k \ge 1$, and $0 \le m \le k$. \ Assume that $W_{\alpha}$ is $(k,2m)$--CPD. \ Then $W_{\alpha}$ is $2m$--contractive, $(2m+2)$--contractive, \ldots, $2k$--contractive.
\end{proposition}

\begin{theorem}
Let $W_{\alpha}$ be a contractive weighted shift whose weight sequence $\alpha$ has a limit (for example, if $W_{\alpha}$ is hyponormal), and fix $m \in \mathbb{N}$. \  Assume that for all $k \geq m$, $W$ is $(k,2m)$--PD or $(k,2m)$--CPD. \ Then $W$ is subnormal.
\end{theorem}

As a consequence, we obtain a result including what we believe is a new sufficient condition for subnormality. 

\begin{proposition} \label{propnew2}
Let $W_{\alpha}$ be a weighted shift with moment sequence $\gamma = (\gamma_n)_{n=0}^\infty$. \ The following statements are equivalent. \newline
(i) $W_{\alpha}$ is $\mathcal{MID}$. \newline
(ii) The sequence $(\delta_n)_{n=0}^\infty$ with $\delta_n = \ln\left(\frac{\gamma_n \gamma_{n+2}}{\gamma_{n+1}^2}\right)$ is a Stieltjes moment sequence. \newline
(iii) The weighted shift with moments $\left(\frac{\delta_n}{\delta_0}\right)$ is subnormal.
\end{proposition}

We now focus on rank-one perturbations of the Agler shifts (see Subsection \ref{Aglershifts} for the definition); our study reveals some remarkable cutoffs for $(k,2m)$--positive definiteness, and for the related weights-squared sequence to be $m$--alternating (for the terminology, see Definition \ref{def:kmpositive}, and page \pageref{defn}, resp.; in particular, $(k,0)$--positive definiteness corresponds to $k$--hyponormality).

\begin{theorem} \label{thm14A}
Let $A_j(x)$ be the perturbation of the Agler shift $A_j$ in which the zeroth weight $(\alpha^{(j)})_0:= \sqrt{\frac{1}{j}}$ is replaced by $(\alpha^{(j)})_0(x) := \sqrt{\frac{x}{j}}$, and let $k \in \mathbb{N}$ and $m \in \{0,1,\ldots,k\}$. \newline
(i) $A_j(x)$ is  $(k,2m)$--PD if and only if $x \leq p(j,k,m)$, where
\begin{equation}
p(j,k,m) \!=\! \ddfrac{(k+1-m)(j+k+m-1)}{k^2 + j k + 2m - j m - m^2}\!=\!\ddfrac{(j+k+m)(k-m)+2m+j-1}{(j+k+m)(k-m)+2m}.
\end{equation}
(ii) \ $A_j(x)$ is $(k,2m)$--PD if and only if $A_j(x)$ is $((j+k+m)(k-m)+2m)$--contractive. 
\end{theorem}

\begin{proposition}
Let $A_j(x)$ be the zeroth weight perturbation of the Agler shift $A_j$ with weight sequence $\alpha^{(j)}$ as in Theorem \ref{thm14A}. \  Then the weights-squared of $A_j(x)$ are $m$--alternating if and only if
\begin{equation}  \label{eq:naltAjofx1}
x \leq 1 + \frac{(j-1) m!}{\prod_{i=1}^{m} (j+i)}.
\end{equation}
If we take $j=2$, which corresponds to the Bergman shift, the weights-squared of $A_2(x)$ are $m$--alternating if and only if $A_2(x)$ is $\frac{(m+1)(m+2)}{2}$--contractive.
\end{proposition}

%%%%%%%%%%%%%%%%%%%%%%%
%%%%%%%%%%%%%%%%%%%%%%%
%%%%%%%%%%%%%%%%%%%%%%%

\section{Notation and Preliminaries} \label{prelim}

\subsection{Unilateral weighted shifts}

To set the notation for weighted shifts, let $\mathbb{N}_0  := \{0, 1, \ldots\}$ and let $\ell^2$ denote the classical Hilbert space $\ell^2(\mathbb{N}_0)$ with canonical orthonormal basis $e_0, e_1, \ldots$ (note that we begin indexing at zero). \  Let $\alpha: \alpha_0, \alpha_1, \ldots$ be a (bounded) positive \textit{weight sequence} and  $W_\alpha$  the weighted shift, defined by linearity and $W_\alpha e_j := \alpha_j e_{j+1} \;\; (j \ge 0)$. \ (Weighted shifts can be defined for any bounded sequence $\alpha$; however, for all questions of interest to us, without loss of generality we can assume (as we do) that $\alpha$ is positive.) \ When $\alpha$ is the constant weight sequence $1,1,\ldots$, the resulting (un-weighted) shift is the classical unilateral shift $U_+e_{j}:=e_{j+1} \; (j \ge 0)$. \ The \textit{moments} $\gamma = (\gamma_n)_{n=0}^\infty$ of the shift are defined by $\gamma_0 := 1$ and $\gamma_n := \prod_{j=0}^{n-1} \alpha_j^2$ for $n \geq 1$. \ From  \cite[III.8.16]{Con} and \cite{GW}, a weighted shift $W_\alpha$ is subnormal if and only if it has a \textit{Berger measure}, meaning a probability measure $\mu$ supported on $[0, \|W_\alpha\|^2]$ such that
$$
\gamma_n = \int_0^{\|W_\alpha\|^2} t^n d \mu(t), \hspace{.2in} n = 0, 1, \ldots .
$$

Applying the Cauchy-Schwarz inequality in $L^2(\mu)$ to the monomials $t^{n/2}$ and $t^{(n+2)/2}$ yields $\gamma_{n+1}^2 \le \gamma_n \gamma_{n+2} \; (n \ge 0)$, and hence $\alpha_n^2 \le \alpha_{n+1}^2 \; (n \ge 0)$. \ Recall that an operator $T$ is hyponormal if $T^*T-TT^* \ge 0$, and one computes easily that for a weighted shift $W_{\alpha}$ this is exactly $\alpha_n^2 \le \alpha_{n+1}^2$ for all $n \ge 0$. 

The canonical polar decomposition of $W_{\alpha}$ is $U_+ P_{\alpha}$, where $P_{\alpha}$ denotes the diagonal operators with diagonal entries $\alpha_0,\alpha_1, \ldots$ . \ The {\it Aluthge transform} \cite{Alu} of $W_{\alpha}$ is given by $AT(W_{\alpha}):=\sqrt{P_{\alpha}}U_+ \sqrt{P_{\alpha}}$, and this is the weighted shift with weight sequence $\sqrt{\alpha_0 \alpha_1},\sqrt{\alpha_1 \alpha_2},\sqrt{\alpha_2 \alpha_3},\ldots$ . \ It is easy to see that $AT$ preserves hyponormality, but whether it preserves subnormality is a nontrivial problem, addressed in detail in \cite{Ex2, CuEx, BCE,BCESC}. \ A sufficient condition for subnormality, pointed out in \cite{CuEx}, is the subnormality of $W_{\sqrt{\alpha}}$.

%%%%%%%%%%%%%%%%%%%%%%%
%%%%%%%%%%%%%%%%%%%%%%%
%%%%%%%%%%%%%%%%%%%%%%%

\subsection{Agler shifts} \label{Aglershifts}
We briefly recall a class of subnormal unilateral weighted shifts with Berger measures that can be easily computed. \ The {\it Agler shifts} $A_j$, $j = 1, 2, \ldots$, are those with weight sequence $\sqrt{\frac{n+1}{n+j}}$, $n = 0, 1, \ldots$ . \  (These were used in \cite{Ag} as part of a model theory for hypercontractions.) \ Since $A_1$ is the unilateral shift, its Berger measure is $\delta_{1}$ (the Dirac point mass at $\{1\}$); for $j \ge 2$ the Berger measure of $A_j$ is $d \mu(t)=(j-1)(1-t)^{j-2}dt$ on $[0,1]$. \ The Agler shifts appear naturally as the weighted shifts associated with the rows (and columns) of the weight diagram for the Drury-Arveson $2$--variable weighted shift \cite[Pages 29--30]{CRC}. 

In \cite{Ex2}, G.R. Exner proved that for $j = 2, 3, \ldots$, and $p > 0$, the (Schur) $p$--th power of $A_j$ is subnormal, as is any $m$--th iterated Aluthge transform of $A_j$. \ The proof (which uses monotone function theory) offers no information about the Berger measure of the resulting shift; however, it brings to the fore the significance of complete monotonicity in the study of $\mathcal{MID}$ shifts, discussed in the next subsection.

%%%%%%%%%%%%%%%%%%%%%%%
%%%%%%%%%%%%%%%%%%%%%%%
%%%%%%%%%%%%%%%%%%%%%%%

\subsection{$\mathcal{MID}$ shifts}

It is useful to note, in considering $\mathcal{MID}$ weighted shifts, that raising every weight to the $p$--th power is equivalent to raising every moment to the $p$--th power. \ We refer the reader to \cite{BCE,BCESC} for an initial study of moment infinite divisibility;  this paper constitutes a continuation of that study.

A companion to the class of Agler shifts is the class of {\it homographic shifts}, that is, those weighted shifts denoted $S(a,b,c,d)$ (where $a, b, c, d >0$ and with $a d > b c$), with weights $\sqrt{\frac{a n + b}{c n + d}}$. \ These shifts were defined and studied in \cite{CPY}, together with certain subshifts of such shifts, and their subnormality established. \  Observe as well that if we throw away a finite number of terms at the beginning of a completely alternating (or log completely alternating) sequence, what remains is still in the original class. \ Therefore, some of the results to follow may be generalized easily to restrictions of shifts to the canonical invariant subspaces of finite co-dimension.

\begin{lemma}(\cite{Ex2}) \! \! The Agler shifts are all $\mathcal{MID}$, as are the contractive shifts $S(a,b,c,d)$  (i.e., $a, b, c, d >0$, $a d > b c$ and $a \le c$), and their subshifts.
\end{lemma}

By definition, a weighted shift $W_{\alpha}$ is completely hyperexpansive when the inequalities for $n$--contractivity in (\ref{cond100}) are reversed; that is, 
$$
\sum_{i=0}^n (-1)^i \binom{n}{i} \gamma_{k+i} \le 0, \qquad k = 0, 1, \ldots ,
$$
for all $n \ge 1$. \ For instance, the Dirichlet shift is completely hyperexpansive. \ We recall that a completely hyperexpansive weighted shift $W_{\alpha}$ gives rise to a subnormal weighted shift by forming a new weight sequence $\delta$ where $\delta_j := \frac{1}{\alpha_j}$ for all $j \ge 0$;  further, one cannot necessarily begin with a subnormal shift, and, by taking reciprocals of the weights, generate a completely hyperexpansive shift. \ (See the discussion after \cite[Proposition 6]{At}.) \ Nevertheless, the following result holds.

\begin{lemma} (\cite[Corollary 4.1]{BCE}) \label{lem22} \ Let $W_\alpha$ be a completely hyperexpansive \linebreak weighted shift with positive weight sequence $(\alpha_n)_{n=0}^\infty$. \ Then the weighted shift with weight sequence $(\frac{1}{\alpha_n})_{n=0}^\infty$ is not only subnormal but is $\mathcal{MID}$.
\end{lemma}

Using the definition of $\mathcal{MID}$, and Schur products, one may show that if $W_\alpha$ is $\mathcal{MID}$ then so is $AT(W_\alpha)$. \ We record this and related results in the following.

\begin{lemma} 
(i) (\cite[Corollary 3.6]{BCE}) \ If a contractive weighted shift $W_\alpha$ is $\mathcal{MID}$ then so is $AT(W_\alpha)$. \newline
(ii) (\cite[Theorem 4.13]{BCESC}) \ $AT$ maps the class $\mathcal{MID}$ bijectively onto itself. \newline
(iii) (\cite[Theorem 4.4]{BCESC}) \ Suppose that $W_{\alpha}$ is a contractive weighted shift whose weights $\alpha_j$ approach a limit (as $j \rightarrow \infty$). \ Then $AT(W_{\alpha})$ is $\mathcal{MID}$ if and only if $W_{\alpha}$ is.
\end{lemma}

%%%%%%%%%%%%%%%%%%%%%%%
%%%%%%%%%%%%%%%%%%%%%%%
%%%%%%%%%%%%%%%%%%%%%%%

\subsection{$k$--hyponormality} \label{khypon}

For the reader's convenience, we sketch the $k$--hyponormality approach to subnormality, and give a brief version of this background (see \cite{Cu1} and \cite{CF1} for a full discussion and some of the beginnings of this substantial study). \ It is the Bram-Halmos characterization of subnormality (see \cite{Br}) that an operator $T$ is subnormal if and only if, for every $k = 1, 2, \ldots$, a certain $(k+1) \times (k+1)$ operator matrix $A_n(T)$ is positive. \ For $k \ge 1$, an operator is $k$--hyponormal if this positivity condition holds for $k$. \ For weighted shifts, it is well-known from \cite[Theorem 4]{Cu} that $k$--hyponormality reduces to the positivity, for each $n$, of the $(k+1) \times (k+1)$ Hankel moment matrix $M_{\gamma}(\ell,k)$, where
$$
M_{\gamma}(\ell,k) = \left(
\begin{array}{cccc}
\gamma _{n} & \gamma _{n+1} & \cdots & \gamma _{n+k} \\
\gamma _{n+1} & \gamma _{n+2} & \cdots & \gamma _{n+k+1} \\
\vdots & \vdots & \ddots & \vdots \\
\gamma _{n+k} & \gamma _{n+k+1} & \cdots & \gamma _{n+2k}%
\end{array}
\right). 
$$

%%%%%%%%%%%%%%%%%%%%%%%
%%%%%%%%%%%%%%%%%%%%%%%
%%%%%%%%%%%%%%%%%%%%%%%

\subsection{$n$--contractivity}

Another approach to subnormality, this time for a contractive operator $T$ (that is, $\|T \| \leq 1$), is the Agler-Embry characterization based on the notion of $n$--contractivity. \  For $n \ge 1$, an operator is $n$\textit{--contractive} if
\begin{equation} \label{eq11}
A_n(T):=\sum_{i=0}^n (-1)^i \binom{n}{i} {T^*}^i T^i \geq 0 \quad \; \textrm{(cf. \cite{Ag})}.
\end{equation}
A contractive operator is subnormal if and only if it is $n$--contractive for all positive integers $n$ (cf. \cite{Ag}). \ It is well-known, and follows easily from ${W_\alpha^*}^i W_\alpha^i$ being diagonal, that for a weighted shift it suffices to test this condition on basis vectors and that a weighed shift is $n$--contractive if and only if
\begin{equation} \label{cond100}
\sum_{i=0}^n (-1)^i \binom{n}{i} \gamma_{k+i} \geq 0, \qquad k = 0, 1, \ldots.
\end{equation}
Given a sequence $a = (a_j)_{j=0}^\infty$, let $\nabla$ (the forward difference operator) be defined by
$$
(\nabla a)_j := a_j - a_{j+1},
$$
and the iterated forward difference operators $\nabla^{n}$ by
$$
\nabla^{0} a := a \; \; \textrm{ and } \; \; \nabla^{n} := \nabla (\nabla^{n-1}),
$$
for $n \geq 1$. \ For instance,
\begin{equation} \label{nabla2}
(\nabla^2a)_j=a_j-2a_{j+1}+a_{j+2} \; \; (j \ge 0).
\end{equation}
To ease the notation in some settings, set, for any $n \geq 1$ and $k \geq 0$, 
$$
T_a(n,k) := (\nabla^{n} a)_k = \sum_{i=0}^n (-1)^i \binom{n}{i} a_{i+k}
$$
(where we write simply $T(n,k)$ if no confusion as to the sequence $a$ will arise). \ With a slight abuse of language say that a sequence $a$ is $n$--contractive if $T_a(n,k) \geq 0$ for all $k = 0, 1, \ldots$ . \  

There is alternative language for these and related notions:  a sequence $a$ is $n$\textit{--monotone} if $T_a(n,k) = (\nabla^{n} a)_k \geq 0$ for all $k = 0, 1, \ldots$, $n$\textit{--hypermonotone} if it is $j$--monotone for all $j = 1, \ldots, n$, and \textit{completely monotone} if it is $n$--monotone for all $n = 1, 2, \ldots$. \  \label{defn} Similarly, a sequence is $n$\textit{--alternating} if $T_a(n,k) = (\nabla^{n} a)_k \leq 0$ for all $k = 0, 1, \ldots$, $n$\textit{--hyperalternating} if it is $j$--alternating for all $j = 1, \ldots, n$, and \textit{completely alternating} if it is $n$--alternating for all $n = 1, 2, \ldots$. \  A sequence is $n$\textit{--log monotone} (respectively, \textit{completely log monotone}, $n$\textit{--log alternating}, \textit{completely log alternating}) if the sequence $(\ln a_j)$ is $n$--monotone (respectively, completely monotone, $n$--alternating, completely alternating). (Notice that we define log monotonicity using $n$--monotonicity, so we allow for the sequence to possibly be negative, as in the case of $\ln a$, where $a$ is the sequence of moments of a contractive weighted shift. \ Thus, we allow slightly more generality than in the definition given in \cite[Definition 6.1]{BCR}.)   

%%%%%%%%%%%%%%%%%%%%%%%
%%%%%%%%%%%%%%%%%%%%%%%
%%%%%%%%%%%%%%%%%%%%%%%

\subsection{$n$--hypermonotone and $n$--hyperalternating weighted shifts}

\begin{proposition}\label{nXvsnhyperX}
Suppose $n \in \mathbb{N}$ and $(a_k)_{k=0}^\infty$ is a sequence of real numbers, and in addition, $\lim_{k \rightarrow \infty} a_k$ exists.
\begin{enumerate}
                     \item If $(a_k)_{k=0}^\infty$ is $n$--monotone, it is $n$--hypermonotone.
                     \item If $(a_k)_{k=0}^\infty$ is $n$--log monotone, it is $n$--log hypermonotone.
                     \item If $(a_k)_{k=0}^\infty$ is $n$--alternating, it is $n$--hyperalternating.
                     \item If $(a_k)_{k=0}^\infty$ is $n$--log alternating, it is $n$--log hyperalternating.
                   \end{enumerate}
\end{proposition}

\begin{proof}
We prove only the first assertion as the other proofs are similar. \ If $n = 1$ there is nothing to prove, so assume $n > 1$. \ Suppose for a contradiction that  $(a_k)_{k=0}^\infty$ is $n$--monotone but not $n$--hypermonotone, and let $m$ be the least integer, $1 < m \leq n$, such that the sequence is $m$--monotone but not $(m-1)$--monotone. \  For each $j$ and $k$, recall that $T(j,k)$ is defined by
$$
T(j,k) = \sum_{i=0}^j (-1)^i \binom{j}{i} a_{k+i}.
$$
The positivity of $T(j,k)$  for all $k$ yields $j$--monotonicity. \  Further, it is routine that, for any $j$ and $k$,
\begin{equation}  \label{eq:recAnk}
T(j+1,k) = T(j,k) - T(j,k+1).
\end{equation}
Suppose $k_0$ to be the least $k$ such that
$$T(m-1,k_0) = \delta < 0.$$
Then
$$0 \leq T(m,k_0) = T(m-1,k_0) - T(m-1,k_0+1)$$
forces
$$T(m-1,k_0+1) \leq \delta.$$
Repeating the argument, we have
\begin{equation} \label{ineq:negative}
T(m-1,k) \leq \delta < 0, \quad k \geq k_0.
\end{equation}
If $\lim_{k \rightarrow \infty} a_k = L$, then it is elementary that
$$
\lim_{k \rightarrow \infty} T(m-1,k) = \sum_{i=0}^{m-1} (-1)^i \binom{m-1}{i} L = 0,
$$
which is a contradiction of \eqref{ineq:negative}. \qed
\end{proof}

Observe that in light of the Dirichlet shift $D$ (that is, the standard $2$--isometry whose moment sequence is therefore $2$--monotone but not $1$--monotone), some assumption like existence of the limit of the weight sequence is required.

\begin{corollary}  \label{contnmonoimpnhypermono}
Suppose $(\gamma_k)_{k=0}^\infty$ is the sequence of moments of a weighted shift. \  If the weighted shift is contractive, and $(\gamma_k)_{k=0}^\infty$ is $n$--monotone for some $n \in \mathbb{N}$, then $(\gamma_k)_{k=0}^\infty$ is $n$--hypermonotone. \  Alternatively, if the sequence $(\gamma_k)_{k=0}^\infty$ is both $1$--monotone and $n$--monotone for some $n \in \mathbb{N}$, then the sequence is $n$--hypermonotone.
\end{corollary}

\begin{proof}
In either case the sequence $(\gamma_k)$ is decreasing and bounded below by zero and so has a limit. \qed
\end{proof}

This transfers to the $n$--log-monotone and $n$--log-hypermonotone case, at least in situations of interest to us.

\begin{corollary}   \label{cor:nmonoimpnhypermonoshifts}
Suppose $(\gamma_k)_{k=0}^\infty$ is the sequence of moments of a (bounded) \linebreak weighted shift $W_{\alpha}$. \  If $W_{\alpha}$ is contractive, and its weights have a non-zero limit (in particular, if $W_{\alpha}$ is hyponormal), then for any $n \in \mathbb{N}$, if $(\gamma_k)_{k=0}^\infty$ is $n$--log-monotone, then it is $n$--log-hypermonotone. \  Alternatively, if the shift is hyponormal or its weights have a non-zero limit, and if the sequence $(\gamma_k)_{k=0}^\infty$ is both $1$--log-monotone and $n$--log-monotone for some $n \in \mathbb{N}$, then the sequence is $n$--log-hypermonotone.
\end{corollary}

\begin{proof}
One applies the previous corollary to the sequence $(\log \gamma_k)$, and the only thing to check is that the given conditions yield that the resulting expressions $T(n,k)$ for the sequence of logs actually have limits, which is straightforward combining the terms using rules of logs. \qed
\end{proof}

We leave to the interested reader the appropriate versions of the corollaries for the $n$--alternating and $n$--log alternating case.

Let us record some results on product sequences. \   It is trivial that the Schur (term-wise) product of two $n$--log monotone sequences is again $n$--log monotone (and, in fact, the Schur (term-wise) product of two $n$--log alternating sequences is again $n$--log alternating). \  The good thing holds for $n$--monotone sequences as well.

\begin{lemma}
Suppose $(a_k)_{k=0}^\infty$ and $(b_k)_{k=0}^\infty$ are $n$--hypermonotone sequences. \linebreak \ Then  $(a_k b_k)_{k=0}^\infty$ is $n$--hypermonotone. \  In particular, if they are moment sequences arising from a contractive weighted shift, and if each is both $1$--monotone and $n$--monotone, it follows that  $(a_k b_k)_{k=0}^\infty$ is $n$--hypermonotone.
\end{lemma}

\begin{proof}
The first claim follows immediately from Leibniz's rule of differences:  if $\nabla$ is the usual forward difference operator on sequences, then
\begin{equation} \label{eq:LeibnizRule}
\nabla^{n} (a_k b_k) = \sum_{j=0}^n \binom{n}{j} (\nabla^{n-j}a_{k+j})(\nabla^{j}b_k).
\end{equation}
Since the coefficients are positive and the differences appearing are non-negative, we have the result. \  The second claim follows because we have shown that, in the contractive case, $n$--hypermonotonicity and $n$--monotonicity coincide. \qed
\end{proof}

We end this section with an example where the complete alternating property is crucial in the proof of $\mathcal{MID}$.

\begin{example} (\cite[Example 3.5]{BCE}) \ The weighted shift with weights the sequence $(\alpha_n)$ defined by
$$\alpha_n = 1 + \frac{1}{2} + \frac{1}{3} + \ldots + \frac{1}{n+1} - \ln (n+2), \hspace{.2in}n=0,1, \ldots,$$
or the shift with weights the square roots of these, is $\mathcal{MID}$ and subnormal. \  This is because this sequence is completely alternating (and it increases to the Euler constant $\gamma \approx 0.577\ldots $).
\end{example}
 
%%%%%%%%%%%%%%%%%%%%%%%
%%%%%%%%%%%%%%%%%%%%%%%
%%%%%%%%%%%%%%%%%%%%%%%

\subsection{Conditionally positive definite matrices}

Throughout this paper, and as is common in the literature, the expression positive definite is to be understood as positive semi-definite; that is, a $k \times k$ matrix $A$ is \textit{positive definite} (in symbols, $A \ge 0$) if $A$ is Hermitian and $\bm{v}^{\ast}A\bm{v} \ge 0$, for all vectors $\bm{v} \in \mathbb{C}^k$.  \ Since we are mainly interested in sequences of real numbers, and matrices built from those sequences, we will focus our discussion on matrices over $\mathbb{R}$. \ Amongst canonical examples of $k \times k$ positive matrices we mention the Gram matrices $\bm{v} \bm{v}^{\ast}$, the Hilbert matrix $\left(\ddfrac{1}{i+j-1} \right)_{i,j=1}^k$, the Pascal matrix $\left(\binom{i+j}{i} \right)_{i,j=1}^k$, the Lehmer matrix $\left(\ddfrac{\min \{i,j\}}{\max \{i,j\}} \right)_{i,j=1}^k$, and the GCD matrix $\left(g.c.d.(i,j) \right)_{i,j=1}^k$. \ These matrices have the added property of remaining positive after taking the $p$--th Schur power, for all $p>0$; that is, they are infinitely divisible matrices (cf. \cite{Bh}).    

\begin{remark}
\rm{
It is easy to see that all positive definite $2 \times 2$ matrices are infinitely divisible. \ However, the same is not true of $3 \times 3$ matrices, witness the matrix 
$$
R(a):=\left( 
\begin{array}{ccc}
1 & 1 & a \\
1 & 2 & 1 \\
a & 1 & 1
\end{array}
\right) .
$$
$R(a)$ is positive definite for all $a \in [0,1]$, but $\det R(a)^{1/2}<0$ for small positive values of $a$.
}
\end{remark}

We briefly pause to give a precise definition of conditional positive definiteness; cf. \cite[p. 66]{BCR}. 

\begin{definition}
A Hermitian $k \times k$ complex matrix $A \equiv (a_{ij})_{i,j=0}^k$ (i.e., $a_{ji}=\overline{a}_{ij}$ for all $i,j=0,\ldots,k$) is \textit{conditionally positive definite} (abbreviated CPD) whenever $\bm{c}^{\ast}A\bm{c}:=\sum_{i,j=0}^k a_{ij}\bar{c}_ic_j \ge 0$ for all complex vectors $\bm{c}\equiv(c_0,c_1,\ldots,c_k)$ such that $c_0+c_1+\ldots+c_k=0$. \ An \textit{infinite} (scalar) matrix is CPD if all of its principal minors of finite size are CPD.
\end{definition}

Clearly, a positive definite matrix is CPD. \ When $k=2$, $A$ is CPD if and only if $a_{11}-a_{12}-a_{21}+a_{22} \ge 0$. \ When $k=3$, a straightforward calculation reveals that $A$ is CPD if and only if the $2 \times 2$ matrix
$$
\left(
\begin{array}{cc}
a_{11}-a_{12}-a_{21}+a_{22} \; \; &  a_{12}-a_{13}-a_{22}+a_{23} \\
a_{21}-a_{22}-a_{31}+a_{32} \; \; &  a_{22}-a_{23}-a_{32}+a_{33} 
\end{array}
\right). 
$$
is positive definite. \ A much more general result holds, as we will see below. \ First, we present two elementary observations. 

\begin{remark}
\rm{
Observe that a real Hermitian matrix $A$ is CPD if and only if it satisfies the CPD positivity condition for vectors with real coordinates adding up to zero. \ Here is the argument for the non-trivial direction: given a complex vector $\bm{v}$ whose components sum to zero, we may write $\bm{v} = \bm{x} + {\mathrm{i}}\bm{y}$ where $\bm{x}$ and $\bm{y}$ are real vectors each of whose coordinates must sum to zero. \  Then 
\begin{eqnarray} \label{xiy}
\bm{v}^{\ast} A \bm{v} &=& (\bm{x} + {\mathrm{i}}\bm{y})^{\ast} A (\bm{x} + {\mathrm{i}}\bm{y}) \nonumber \\
 &=& \bm{x}^T A \bm{x} + \bm{y}^T A \bm{y} - {\mathrm{i}}\bm{y}^T A\bm{x} +{\mathrm{i}}\bm{x}^T A \bm{y} \nonumber \\
 &=& \bm{x}^T A \bm{x} + \bm{y}^T A \bm{y} - {\mathrm{i}}(\bm{y}^T A\bm{x} - \bm{x}^T A \bm{y}) \\ 
 &=& \bm{x}^T A \bm{x} + \bm{y}^T A \bm{y}, \nonumber 
\end{eqnarray}
using the fact that $\bm{v}^{\ast} A \bm{v}$ must be real (since $A$ is Hermitian) and therefore the imaginary part of (\ref{xiy}) must be zero. \ If we assume that $A$ satisfies conditional positive definiteness for real vectors, we must have $\bm{x}^T A \bm{x} + \bm{y}^T A \bm{y} \ge 0$, and it follows that $\bm{v}^{\ast} A \bm{v}  \ge 0$, showing that $A$ satisfies conditional positive definiteness for complex vectors.
}
\end{remark}

\begin{remark}
\rm{
(i) \ (cf. \cite[Lemma 4.1.4]{BaRa}) \ If $A$ is CPD, then $A$ has at most one negative eigenvalue. \newline
(ii) \ Assume that a $k \times k$ matrix $A$ is nonzero, has non-positive entries, and is CPD. \ Then $A$ has exactly one negative eigenvalue. \ For, since $A$ is Hermitian, all eigenvalues are real, and not all eigenvalues can be zero. \ Since the trace of $A$ is non-positive, at least one eigenvalue must be negative. \ By (i), at most one eigenvalue can be negative, and therefore $A$ has exactly one negative eigenvalue. \newline
(iii) \ Consider a contractive unilateral weighted shift $W_{\alpha}$, and its moment matrices $M_{\gamma}(0)$ and $M_{\gamma}(1)$, whose entries are real numbers in the interval $(0,1]$ (recall that we have assumed that the weights are positive). \ Then $\log M_{\gamma}(0,k)$ and $\log M_{\gamma}(1,k)$ are nonzero matrices, have non-positive entries, and are CPD. \ It follows from (ii) that each of these matrices has exactly one negative eigenvalue.    
}
\end{remark}

Since we wish to study CPD matrices and weighted shifts, the equivalence of (i) and (ii) in the following result is of special interest to us.

\begin{theorem} \label{thm114}
(cf. \cite[Theorem 4.1.3]{BaRa}, \cite{Bh}, and \cite[Problem 6.3.8]{HJ}) \ For $k \ge 2$, let $A \equiv \left(a_{ij}\right)_{i,j=1}^k$ be a Hermitian $k \times k$ matrix. \ The following statements are equivalent. \newline
(i) \ $A$ is CPD. \newline
(ii) \ $B := \left(a_{ij}-a_{i,j+1}-a_{i+1,j}+a_{i+1,j+1}\right)_{i,j=1}^{k-1}$ is positive definite. \newline
(iii) \ There exist real numbers $\omega_i \; (i=1,\ldots,k)$ such that $C:=\left(a_{ij}-\omega_i-\omega_j \right)_{i,j=1}^k$ is positive definite. \newline
(iv) \ For all $p>0$, the $k \times k$ matrix $e^{pA}$ is positive definite.
\end{theorem}

\begin{remark} \label{rem115} 
\rm{
Assume that the matrix $A$ in Theorem \ref{thm114} is Hankel, so the entries are of the form $a_{ij}=w_{i+j}$. \ Then the entries of the matrix $B$ in {\it (ii)} are of the form $w_{i+j}-2w_{i+j+1}+w_{i+j+2}$. \ This sum is precisely what one obtains by letting the square of the forward difference operator $\nabla$ act on $w$ (cf. (\ref{nabla2})). \ We'll revisit this viewpoint in Subsection \ref{condition}. 
}
\end{remark}

We now establish a link between conditional positive definiteness and unilateral weighted shifts; cf. Lemma \ref{lem22}.

\begin{remark}
\rm{
Let $A$ be a $k \times k$ matrix with positive entries, and assume that -A is CPD. \ Then the $k \times k$ matrix $R$ of reciprocals, defined as $r_{ij}:=\ddfrac{1}{a_{ij}}$, is infinitely divisible \cite[Subsection 6.3, Problem 11]{HJ}.  
}
\end{remark}
    
We now present a proposition crucial for our purposes, which R. Bhatia \cite{Bh} attributes to C. Loewner and which may be found in \cite{HJ}.

\begin{proposition} (\cite[Theorem 6.3.13]{HJ}) \ \label{prop215}
Suppose $M$ is a symmetric matrix with positive entries. \  Then $M$ is infinitely divisible if and only if $\log M$ (meaning the Hadamard logarithm matrix) is CPD.
\end{proposition}

\begin{remark}
\rm{
CPD matrices are intrinsically associated with \textit{distance} matrices, as follows. \ In $\mathbb{R}^d$, consider $k$ points $\bm{p}^{(1)},\ldots,\bm{p}^{(k)}$ and their Euclidean distances $\left\|\bm{p}^{(i)}-\bm{p}^{(j)}\right\|^2 \; (i,j=1,\ldots,k)$. \ Now form the $k \times k$ matrix $D(\bm{p}^{(1)},\ldots,\bm{p}^{(k)}) \equiv \left(d_{ij}\right)_{i,j=1}^k$ where $d_{ij}(\bm{p}^{(1)},\ldots,\bm{p}^{(k)}):=\left\|\bm{p}^{(i)}-\bm{p}^{(j)}\right\|^2$, for $i,j=1,\ldots,k$; $D \equiv D(\bm{p}^{(1)},\linebreak \ldots,\bm{p}^{(k)})$ is called the \textit{distance} matrix associated with the $k$--tuple $(\bm{p}^{(1)},\ldots,\bm{p}^{(k)})$. \ Clearly, $d_{ii}=0$ for $i=1,\ldots,k$, so the diagonal entries of $D$ are all zero. \ Moreover, $-D$ is CPD, as a calculation reveals. \ In fact, fix $i$ and $j$ and write $\left\|\bm{p}^{(i)}-\bm{p}^{(j)}\right\|^2=\left\|\bm{p}^{(i)}\right\|^2+\left\|\bm{p}^{(j)}\right\|^2-2 \left\langle \bm{p}^{i},\bm{p}^{j}\right\rangle$. \ Now let $\omega_i:=\left\|\bm{p}^{(i)}\right\|^2$, $m_{ij}:=\omega_i+\omega_j$, $M:=(m_{ij})_{i,j=1}^k$, and $g_{ij}:=\left\langle \bm{p}^{i},\bm{p}^{j}\right\rangle$. \ 
For $k \ge 2$, consider the map from Hermitian $k \times k$ matrices to Hermitian $(k-1) \times (k-1)$ matrices, given by $A \longmapsto \Phi(A):=(a_{ij}-a_{i,j+1}-a_{i+1,j}+a_{i+1,j+1})_{i,j=1}^{k-1}$. \ The above mentioned matrix $M$ is a typical element of the linear subspace of $k \times k$ matrices $A$ such that $\Phi(A)=0$; that is, $M \in \ker \Phi$. \ It follows that $D$ can be written as the difference between a matrix in $\ker \Phi$ and twice the Gramian matrix $G:=(g_{ij})_{i,j=1}^k$. \ 
}
\end{remark}

Now consider an arbitrary CPD Hermitian $k \times k$ matrix $A \equiv (a_{ij})_{i,j=1}^k$ and let $H_A:=(h_{ij})_{i,j=1}^k$, with $h_{ij}:=(a_{ii}+a_{jj})/2 \quad (i,j=1,\ldots,k)$. \ We already know that $\Phi(H_A) = 0$, so $H_A \in \ker \Phi$. \ Moreover, $A-H_A$ is CPD and has diagonal identically equal to zero; then, by \cite[Theorem 4.1.7]{BaRa}, $-(A-H_A)$ is a distance matrix. \ From the discussion above, we see that 
$$
-(A-H_A)=M-2G,
$$
for some $M \in \ker \Phi$ and $G$ Gramian. \  
As a result, we can prove that, given a CPD Hermitian $k \times k$ matrix $A$, one can find a positive integer $d$, and $k$ points $\bm{p}^{(i)} \in \mathbb{R}^d \; (i=1,\ldots,k)$, such that $A-2G \in \ker \Phi$, where $G$ is the Gramian of the $k$--tuple $(\bm{p}^{(1)},\ldots,\bm{p}^{(k)})$. \ In particular, $\Phi(A) = 2 \Phi(G) \ge 0$, since it is well-known that Gramians are PD (see, for instance, \cite[Theorem 7.2.10]{HJ2}).

We close this section by mentioning a very recent paper of Z.J. Jab\l o\'nski, I.B. Jung and J. Stochel \cite{JJS}, in which they study the CPD condition for operators on Hilbert space. \ The focus is on CPD sequences of exponential growth, with emphasis given to obtaining criteria for their positive definiteness.
%%%%%%%%%%%%%%%%%%%%%%%
%%%%%%%%%%%%%%%%%%%%%%%
%%%%%%%%%%%%%%%%%%%%%%%
%%%%%%%%%%%%%%%%%%%%%%%
%%%%%%%%%%%%%%%%%%%%%%%
%%%%%%%%%%%%%%%%%%%%%%%

\section{A Bridge Between $k$--hyponormality and $n$--contractivity} \label{main}

Proposition \ref{prop215} establishes an equivalence between the infinite divisibility of a moment matrix $M$ and the conditional positive definiteness of its Schur logarithm $\log M$. \ This allows us to view $k$--hyponormality under a new lens, and describe moment infinite divisibility of a weighted shift $W_{\alpha}$ as follows; for each $k \ge 1$, $i \ge 0$, and $p >0$, the matrix
$$M_{\gamma}^p(i,k) = \begin{pmatrix}
\gamma_i^p & \gamma_{i+1}^p &\gamma_{i+2}^p  &\ldots  & \gamma_{i+k}^p\\
 \gamma_{i+1}^p & \gamma_{i+2}^p &  \ldots & \ldots & \gamma_{i+k+1}^p \\
 \gamma_{i+2}^p & \ldots & \ldots & \ldots  & \gamma_{i+k+2}^p  \\
\vdots  & \vdots & \vdots & \ddots & \vdots \\
\gamma_{i+k}^p & \gamma_{i+k+1}^p & \gamma_{i+k+2}^p  & \ldots & \gamma_{i+2k}^p \\
\end{pmatrix}$$
is positive definite, where the $\gamma_i$ are the moments of the shift. \  (When $p=1$, we will simply write $M_{\gamma}(i,k)$.)

The following two corollaries are immediate. \ Note, however, that in this context we do not need to assume that $W_\alpha$ is a contraction: while the Agler conditions (equivalently, the approach through monotone moment sequences) yield subnormality for a contraction (but not in general, viz. the Dirichlet shift), the $k$--hyponormality conditions do yield subnormality for any shift. 

Observe also that if we wish to operate at the level of infinite matrices of moments, we must take care that we examine two such matrices. \ Recall that for subnormality we are in pursuit of \textit{Stieltjes} moment sequences, that is, moment sequences $(\gamma_n)_{n=0}^\infty$ for which the following two matrices are simultaneously positive definite \cite{Ak}:
\begin{equation} \label{def:MofW}
M_{\gamma}(0) \equiv M_{\gamma}(0,\infty):= \left(
\begin{array}{ccccc}
\gamma _{0} & \gamma _{1} & \cdots & \gamma _{k} & \ldots \\
\gamma _{1} & \gamma _{2} & \cdots & \gamma _{k+1} & \ldots\\
\vdots & \vdots & \ddots & \vdots & \vdots\\
\gamma _{k} & \gamma _{k+1} & \cdots & \gamma _{k+2n} & \ldots \\
\vdots & \vdots & \ddots & \vdots &  \ldots
\end{array}
\right)
\end{equation}
and 
\begin{equation} \label{def:M1ofW}
M_{\gamma}(1) \equiv M_{\gamma}(1,\infty): = \left(
\begin{array}{ccccc}
\gamma _{1} & \gamma _{2} & \cdots & \gamma _{k+1} & \ldots \\
\gamma _{2} & \gamma _{3} & \cdots & \gamma _{k+2} & \ldots\\
\vdots & \vdots & \ddots & \vdots & \vdots\\
\gamma _{k+1} & \gamma _{k+2} & \cdots & \gamma _{k+2n+1} & \ldots \\
\vdots & \vdots & \ddots & \vdots &  \ldots
\end{array}
\right).
\end{equation}

\begin{corollary}
Let $W_\alpha$ be a weighted shift. \ Then $W_\alpha$ is $\mathcal{MID}$ if and only if for each $i \geq 0$ and each $k \geq 1$, 
$\log M_{\gamma}(i,k)$ is CPD (equivalently, both $\log M_{\gamma}(0)$ and $\log M_{\gamma}(1)$ are CPD).
\end{corollary}

\begin{corollary}     \label{cor:infDivkHN}
Let $W_\alpha$ be a weighted shift and let $k \in \mathbb{N}$. \ Then $W_\alpha^p$ is $k$--hyponormal for all $p>0$ if and only if 
$\log M_{\gamma}(i,k)$ is CPD for each $i \geq 0$.
\end{corollary}

We leave to the interested reader to show that the corollaries may be improved to add the equivalent condition ``subnormal for each $p$ in some interval $[0, \delta)$ with $\delta > 0$'' (respectively, ``$k$--hyponormal for each $p$ in some interval $[0, \delta)$ with $\delta > 0$''). \  This may be accomplished by considering positivity of Schur products of positive matrices.

We may apply this to a weighted shift to yield the following result. \ This yields a test for subnormality which is, as far as we know, new.

\begin{proposition} \label{propnew1}
Let $W_{\alpha}$ be a weighted shift with moment sequence $\gamma = (\gamma_n)_{n=0}^\infty$. \ The following statements are equivalent. \newline
(i) $W_{\alpha}$ is $\mathcal{MID}$. \newline
(ii) The sequence $(\delta_n)_{n=0}^\infty$ with $\delta_n = \ln\left(\frac{\gamma_n \gamma_{n+2}}{\gamma_{n+1}^2}\right)$ is a Stieltjes moment sequence. \newline
(iii) The weighted shift with moments $\left(\frac{\delta_n}{\delta_0}\right)$ is subnormal.
\end{proposition}

\begin{proof}
This is immediate from the fact that $\log M_{\gamma}(0)$ and $\log M_{\gamma}(1)$ must be CPD, Theorem \ref{thm114} (i) $\Leftrightarrow$ (ii), and the standard two-matrix requirement that a sequence be a Stieltjes moment sequence.  \qed
\end{proof}

When the weighted shift is a contraction we may add to the list, in Proposition \ref{propnew1}, of conditions equivalent to moment infinite divisibility, as we do in the following theorem. \  Note that, however, as in the observation following Corollary \ref{cor:infDivkHN}, we may instead add a condition only on some Schur powers $p$ for $p$ in a small interval $[0, \delta)$. \ To extend this to all $p>0$, one needs to resort to Schur multiplication to go from say, $p$ to $2p$, and so on.

We now prove our first main result, which we restate for the reader's convenience.

\begin{theorem}  \label{thm:equivcondtIDcontraction2}
Let $W_\alpha$ be a contractive weighted shift, and let $M_{\gamma}(i,k)$, $M_{\gamma}(0)$ and $M_{\gamma}(1)$ be as above. \ The following statements are equivalent.
\begin{enumerate}
             \item $W_\alpha$ is $\mathcal{MID}$.
             \item $\log M_{\gamma}(0)$ and $\log M_{\gamma}(1)$ are CPD.
						 \item For every $i \ge 0$ and every $k \ge 1$, $\log M_{\gamma}(i,k)$ is CPD.
             \item For every $p>0$, $M_{\gamma}^p(0)$ and $M_{\gamma}^p(1)$ are positive definite.
             \item For every $p>0$, $M_{\gamma}^p(0)$ and $M_{\gamma}^p(1)$ are CPD.
             \item The moment sequence $\gamma$ is log completely monotone.
             \item The weight sequence $\alpha$ is log completely alternating.
\end{enumerate}
\end{theorem}

\begin{proof}
The only assertion not immediate is the equivalence of $M_{\gamma}^p(0)$ and $M_{\gamma}^p(1)$ being CPD (for all $p>0$) to the others. \ Fix some power $p>0$ for the moment. \ Using the observation just before the theorem, the condition in question yields all of the $2k$--monotonicity conditions, and, using Corollary \ref{cor:nmonoimpnhypermonoshifts} we obtain $2k$--hypermonotonicity for all $k$, yielding that the sequence of moments $\gamma^{p}$ is completely monotone. \  Since $p$ was arbitrary, this yields the result. \qed
\end{proof}

In view of the results in this section, one can try and extend the notion of moment infinite divisibility for weighted shifts to arbitrary contractions on Hilbert space, using the polar decomposition factor, as in Problem \ref{problem1} below. \ We plan to pursue this idea in a forthcoming paper.

\begin{problem} \label{problem1}
Starting with the canonical polar decomposition of a subnormal operator, $T \equiv V\left|T\right|$, should one consider the subnormality of $V\left|T\right|^p$ for all $p>0$ as the proper analog of moment infinite divisibility?
\end{problem}

%%%%%%%%%%%%%%%%%%%%%%%%%%%%%%%
%%%%%%%%%%%%%%%%%%%%%%%%%%%%%%%
%%%%%%%%%%%%%%%%%%%%%%%%%%%%%%%

\subsection{A unifying condition} \label{condition}

It is well-known that $k$--hyponormality for all $k = 1, 2, \ldots$ is necessary and sufficient for subnormality of an operator;  it is true as well that $n$--contractivity for all $n = 1, 2, \ldots$ is sufficient for subnormality of an operator, and necessary if the operator is a contraction. \  Of course these, in their equivalents as conditions on moments of the shift, hold as well. \  With the exception of \cite[Theorem 1.2]{EJP}, which shows in general that $k$--hyponormality implies $2k$--contractivity, and work in \cite{AE} which gives, in a very special case, a situation in which each $k$--hyponormality coincides exactly with a particular $n$--contractivity, the two sets of conditions have remained resistant to a common point of view. \  It is the goal of this subsection to give a more general condition exhibiting a unified point of view yielding the sets of conditions.

We will have occasion to operate mainly upon a sequence likely in the sequel to be the moment sequence of a weighted shift (often contractive), so we will employ $\gamma = (\gamma_n)_{n=0}^\infty$ for this sequence. \  Recall the forward difference operator $\nabla$ acting on sequences by $\nabla(\gamma)_n = \gamma_n - \gamma_{n+1}$ for all $n = 0, 1, \ldots$, and with its powers defined iteratively as $\nabla^{m+1}(\gamma) := \nabla(\nabla^{m}(\gamma))$ for all $m \geq 0$; we let $\nabla^{0}(\gamma) := \gamma$. \  Note that $m$--contractivity means that $\nabla^{m}(\gamma)$ is a non-negative sequence.

For a sequence $a = (a_n)_{n=0}^\infty$, $k \geq 1$, and $i \geq 0$, we let $M_a(i,k)$ denote the Hankel matrix of size $k+1$ by $k+1$ given by
$$M_a(i,k) := \begin{pmatrix}
a_i & a_{i+1} &a_{i+2}  &\ldots  & a_{i+k}\\
 a_{i+1}& a_{i+2} &  \ldots & \ldots & a_{i+k+1} \\
a_{i+2} & \ldots & \ldots & \ldots & a_{i+k+2} \\
\vdots  & \vdots & \vdots & \ddots & \vdots \\
a_{i+k} & a_{i+k+1} & a_{i+k+2} & \ldots & a_{i+2k} \\
\end{pmatrix}.$$
(Observe that we use the positive integer $k$ to keep track of $k$--hyponormality, but the actual size of the related matrices is $k+1$, as in Subsection \ref{khypon}.) \ In the case of a weighted shift, $k$--hyponormality is exactly that $M_\gamma(i,k)$ is positive definite for each $i = 0, 1, \ldots$. \  Observe that by  Theorem \ref{thm114} (i) $\Leftrightarrow$ (ii) applied to Hankel matrices, it follows that $M_\gamma(i,k)$ is a CPD matrix if and only if $M_{\nabla^{2}(\gamma)}(i,k-1)$ is positive definite; see also Remark \ref{rem115}.

\begin{definition}  \label{def:kmpositive}
Let $k \ge 1$ and $0 \le m \le k$. \ We say a sequence $\gamma = (\gamma_n)_{n=0}^\infty$ is $(k,2m)$--PD if, for every $i \ge 0$, the matrix $M_{\nabla^{2m}(\gamma)}(i, k-m)$ is positive definite. \  Similarly, we say the sequence is $(k,2m)$--CPD if, for every $i \ge 0$, the matrix $M_{\nabla^{2m}(\gamma)}(i, k-m)$ is CPD.
\end{definition}

We will abuse language slightly to say a weighted shift $W_\alpha$ with moment sequence $\gamma$ is $(k,2m)$--PD (respectively, $(k,2m)$--CPD) if its moment sequence is. \  Observe that a weighted shift is $k$--hyponormal exactly when it is $(k,0)$--PD;  a computation shows that a shift is $(k,2k)$--PD if and only if it is $2k$--contractive. \  This is the sense in which the $k$--hyponormality condition and the $2k$--contractivity condition are simply instances (in fact, extremes) of the single more general condition. \  It is productive to think of the condition as beginning with the moment matrix of size $k+1$, employing repeatedly the matrix transformation implicit in Theorem \ref{thm114}\,(ii), and insisting that the matrix obtained after $m$ such steps be positive definite (respectively, CPD).

Observe also that if the shift $W_\alpha$ is a contraction, with some other mild condition (such as hyponormality), and if $W_\alpha$ is $(k,2k)$--PD, then its moment sequence is $2k$--hypermonotone (see Corollary  \ref{contnmonoimpnhypermono}).

We now record the following proposition.

\begin{proposition}       \label{prop:getAglerfrom}
Let $W_{\alpha}$ be a unilateral weighted shift, $k \ge 1$, and $0 \le m \le k$. \ Assume that $W_{\alpha}$ is $(k,2m)$--CPD. \ Then $W_{\alpha}$ is $2m$--contractive, $(2m+2)$--contractive, \ldots, $2k$--contractive.
\end{proposition}

\begin{proof}
This is just a computation:  use each matrix $M_{\nabla^{2m}(\gamma)}(i, k-m)$ as a quadratic form against an appropriately sized ``negative binomial'' vector, such as $(0,0,1,-1,0)$ or $(0, 1, -3, 3, -1, 0, 0)$. \ To show one significant instance of this calculation, let $k=4$, $m=2$ and $i=\ell$. \ Then $M_{\nabla^2 \gamma}(\ell,k-m)$ becomes
$$
M_{\nabla^2 \gamma}(\ell,2) \!=\!\! \left( 
\begin{array}{ccc}
\gamma_{\ell}-2\gamma_{\ell+1}+\gamma_{\ell+2} \; & \gamma_{\ell+1}-2\gamma_{\ell+2}+\gamma_{\ell+3} \; & \gamma_{\ell+2}-2\gamma_{\ell+3}+\gamma_{\ell+4} \\
\gamma_{\ell+1}-2\gamma_{\ell+2}+\gamma_{\ell+3} \; & \gamma_{\ell+2}-2\gamma_{\ell+3}+\gamma_{\ell+4} \; & \gamma_{\ell+3}-2\gamma_{\ell+4}+\gamma_{\ell+5} \\
\gamma_{\ell+2}-2\gamma_{\ell+3}+\gamma_{\ell+4} \; & \gamma_{\ell+3}-2\gamma_{\ell+4}+\gamma_{\ell+5} \; & \gamma_{\ell+4}-2\gamma_{\ell+5}+\gamma_{\ell+6} 
\end{array}
\right)\!\!,
$$
and therefore
$$
\left(
\begin{array}{ccc}
 1 & -1 & 0
\end{array}
\right) M_{\nabla^2 \gamma}(\ell,2) \left(
\begin{array}{c}
1 \\
-1 \\
0
\end{array}
\right) = \gamma_{\ell} - 4 \gamma_{\ell+1} + 6 \gamma_{\ell+2} - 4 \gamma_{\ell+3} + \gamma_{\ell+4} , 
$$
which is one of the Agler expressions needed to verify $4$--contractivity. \qed
\end{proof}

%%%%%%%%%%%%%%%%%%%%%%%%%%%%%%%
%%%%%%%%%%%%%%%%%%%%%%%%%%%%%%%
%%%%%%%%%%%%%%%%%%%%%%%%%%%%%%%

\subsection{A family of intermediate ladders}

It is well-known that the $k$--hyponormality conditions ($k \in \mathbb{N}$) provide a \linebreak ``ladder'' with subnormality at the top, and that the $n$--contractivity conditions $n=1, 2, \ldots$ serve similarly if the shift is a contraction. \  If $W$ is a contraction whose weights have a limit (for example, it is hyponormal), and if we fix some $m$, the $(k,2m)$--PD or $(k,2m)$--CPD conditions (for $k \geq m$) form a similar ladder. \  The proof is immediate from Proposition \ref{prop:getAglerfrom}, the fact that (in the presence of contractivity) $n$--monotonicity implies $n$--hypermonotonicity, and the Agler-Embry characterization of subnormality as applied to a shift.

\begin{theorem}
Let $W$ be a contractive weighted shift whose weights have a limit (for example, if $W$ is hyponormal), and fix $m \in \mathbb{N}$. \  If for all $k \geq m$, $W$ is $(k,2m)$--PD or $(k,2m)$--CPD, $W$ is subnormal.
\end{theorem}

We turn next to some results showing that these ladders are ``different'' by considering zeroth-weight perturbations $A_j(x)$ of the Agler shifts $A_j$. \  Recall that the moment sequence $\gamma^{(j)} = (\gamma_n^{(j)})_{n=0}^\infty$ of the $j$--th Agler shift satisfies
\begin{equation}   \label{eq:Aglerjmoments}
\gamma_n^{(j)} = \frac{(j-1)!\!\;n!}{(n+j-1)!} \; \; \; (j \ge 2, n \ge 0).
\end{equation}
It is a computation to show that
$$\nabla(\gamma^{(j)}) = \frac{j-1}{j}\cdot\gamma^{(j+1)}$$
and that therefore
\begin{equation}  \label{eq:nabla2mgammajIS}
\nabla^{2m}(\gamma^{(j)}) = \frac{j-1}{j+2m - 1}\cdot \gamma^{(j+2m)}.
\end{equation}
Observe that this is in general not a moment sequence, because $(\nabla^{2m}(\gamma^{j}))_0$ may not be equal to $1$. \ However, since $A_j$ is subnormal for all $j$, $\nabla^{2m}(\gamma^{j})$ is, up to a normalizing constant, a Stieltjes moment sequence.

%%%%%%%%%%%%%%%%%%%%%%%%%%
%%%%%%%%%%%%%%%%%%%%%%%%%%
%%%%%%%%%%%%%%%%%%%%%%%%%%

\bigskip

\subsection{Cutoffs for rank-one perturbations of the Agler shifts} \label{cutoffs}

Recall that if we form the standard Hankel matrices of moments of a weighted shift, $n$--hyponormality amounts to positivity of all those matrices of size $n+1$ by $n+1$.  We have seen as well that instead of insisting these matrices be positive, we might ask only that the matrices of size  $n+1$ by $n+1$ be CPD, which is equivalent to the matrices of size $n$ by $n$, and consisting of second differences $(\nabla^{2})$, be positive definite.   And of course we could insist only that these matrices be CPD, which is equivalent to matrices of size $n-1$ by $n-1$, and consisting of fourth differences of the original moments, be positive definite.  We may continue in this way, terminating by insisting that matrices of size $1$ by $1$, and having the entry consisting of $2n$--th order differences, be positive definite, and this turns out just to be the $2n$--th Agler condition.

We set some notation, so if the original moment sequence is $\gamma = (\gamma_n)$, we write $M_{\nabla^{2m}\gamma}(k,n-1)$ for the matrix of $2m$--th order differences of $\gamma$, of size $n$ by $n$, and with the upper-left-hand entry consisting of the difference beginning at $\gamma_k$. \ However, in what follows we need to consider only those matrices for which $k = 0$, and we abbreviate $M_{\nabla^{2m}\gamma}(0,n-1)$
to $(\nabla^{2m} \gamma)_{n}$. \ Thus, for example, 
$$(\nabla^{4} \gamma)_{3} = \left(\begin{array}{ccc}
\gamma_0 - 4 \gamma_1 + \ldots & \gamma_1 - 4 \gamma_2 + \ldots & \gamma_2 - 4 \gamma_3 + \ldots \\
\gamma_1 - 4 \gamma_2 + \ldots & \gamma_2 - 4 \gamma_3 + \ldots & \gamma_3 - 4 \gamma_4 + \ldots \\
\gamma_2 - 4 \gamma_3 + \ldots & \gamma_3 - 4 \gamma_4 + \ldots & \gamma_4 - 4 \gamma_5 + \ldots \\
\end{array}\right).$$

We shall consider perturbations $A_j(x)$ of the $j$--th Agler shift $A_j$ in which the zeroth weight $\alpha_0 = \sqrt{\frac{1}{j}}$ is replaced by $\alpha_0(x) = \sqrt{\frac{x}{j}}$. \  It is well-known that this results in the moment sequence $1, x \gamma^{(j)}_1, x \gamma^{(j)}_2, x \gamma^{(j)}_3, \ldots$. \  The task is to consider the cutoffs in $x$ for which $A_j(x)$ is various $(k,2m)$--PD; as usual with zeroth-weight perturbations, we need to consider only the tests of matrices whose $(1,1)$ entry is $\gamma_0^{(j)} + \ldots$ since other matrices relevant to some $(k,2m)$--PD are simply $x$ times a matrix relevant to $A_{j+2m}$, which is subnormal and so the matrix is positive definite (and hence CPD). \  To detect positive definiteness, we will use the nested determinant test. \ Perturbations $A_2(x)$ of the Bergman shift were studied initially in \cite[Proposition 7]{Cu}, where the specific cutoff for $2$--hyponormality was calculated. \ Here we will follow the approach described in \cite[Section 2]{AE}, both for $k$--hyponormality and $n$--contractivity of Agler-type shifts. \ We will therefore merely sketch the computations here. \ For the reader's convenience, below we state one of the main results in \cite{AE}.

\begin{theorem} \label{AdamsExner} 
(\cite[Theorem 2.6]{AE}) \ For $j \ge 2$, let $A_j(x)$ be the rank-one perturbation of the $j$--th Agler shift, as described above. \newline
(i) \ $A_j(x)$ is $n$--contractive if and only if $x \le c(j,n):=\ddfrac{n+j-1}{n}$. \newline
(ii) \ $A_j(x)$ is $k$--hyponormal if and only if $x \le h(j,k):=\ddfrac{k(k+j)+j-1}{k(k+j)}$. \newline 
As a consequence, $A_j(x)$ is $k$--hyponormal if and only if $A_j(x)$ is $k(k+j)$--contractive.  
\end{theorem}

\begin{remark}
\rm{
(i) \ The quantities on the right-hand side of the inequalities in Theorem \ref{AdamsExner}(i) and (ii) are consistent with previous partial results on rank-one perturbations of weighted shifts (see \cite[Proposition 7]{Cu} and \cite[Problem 3.2 and Proposition 3.4]{CF1}); however, while those results define the $0$--th weight as $x$, for our purposes it is better to write it as $\sqrt \ddfrac{x}{j}$. \newline
(ii) \ In the specific case of the Bergman shift ($j=2$), we get $c(2,n)=\ddfrac{n+1}{n}$ and $h(2,k)=\ddfrac{k^2+2k+1}{k^2+2k}$, so it is immediate that $A_2(x)$ is $2$--hyponormal if and only if it is $8$--contractive.   
}
\end{remark}
 
Our aim is to obtain a comparable result for $(k,2m)$--PD, which will enable us to relate this notion to $n$--contractivity. \ To ease the notation, view $j$ as fixed for the moment (which we then suppress as much as possible) and let $N(k,m)$ denote the matrix $M_{\nabla^{2m}(\gamma^{j})}(0, k-m)$ and let $\hat{N}(k,m)$ denote the matrix obtained from $N(k,m)$ by deleting the first row and first column. \  To detect $(k,2m)$--PD we wish to test, for positivity, the determinant $\det M_{\nabla^{2m}(\gamma^{j}(x))}(0, k-m)$;  it is a computation to show that this positivity is equivalent to 
$$
x \leq \frac{1}{1 - \frac{\det N(k,m)}{\det \hat{N}(k,m)}}.
$$
In light of \eqref{eq:nabla2mgammajIS} and the fact that $N(k,m)$ is of size $k+1-m$, we have
$$\det N(k,m) = \left(\frac{j-1}{j+2m - 1}\right)^{k+1-m} \det M_{\gamma^{(j+2m)}}(0,k-m)$$
and using again \eqref{eq:nabla2mgammajIS} and that $\hat{N}(k,m)$ is of size $k+1-m-1 = k-m$, we obtain
$$\det \hat{N}(k,m) = \left(\frac{j-1}{j+2m - 1}\right)^{k-m} \det M_{\gamma^{(j+2m)}}(2,k-m-1).$$
The ratio $\frac{\det M_{\gamma^{(j+2m)}}(0,k-m)}{\det M_{\gamma^{(j+2m)}}(2,k-m-1)}$ was computed in \cite[Lemma 2.5]{AE};  transferred to our setting, it is
$$\frac{j + 2m - 1}{(k-m + j + 2m-1)(k-m+1)}.$$

Putting this all together yields the following result.

\begin{theorem} \label{thm14}
Let $A_j(x)$ be the perturbation of the Agler shift $A_j$ in which the zeroth weight $\alpha_0 = \sqrt{\frac{1}{j}}$ is replaced by $\alpha_0(x) = \sqrt{\frac{x}{j}}$, and let $k \in \mathbb{N}$ and \linebreak $m \in \{0,1,\ldots,k\}$. \newline
(i) $A_j(x)$ is  $(k,2m)$--PD if and only if $x \leq p(j,k,m)$, where
\begin{equation}
p(j,k,m) = \ddfrac{(k+1-m)(j+k+m-1)}{k^2 + j k + 2m - j m - m^2}=\ddfrac{(j+k+m)(k-m)+2m+j-1}{(j+k+m)(k-m)+2m}. 
\end{equation}
(ii) \ $A_j(x)$ is $(k,2m)$--PD if and only if $A_j(x)$ is $((j+k+m)(k-m)+2m)$--contractive. 
\end{theorem}

We now present some examples of this result for the case of the Bergman shift ($j=2$). \  Recall that the size of the relevant matrix is $k+1$ (in accordance with the size of the matrix to be tested for $k$--hyponormality). \ In this case, $p(j,k,m)=p(2,k,m)=\ddfrac{k(k+2)-m^2+1}{k(k+2)-m^2}$; observe that the numerator is 1 plus the denominator.
$$
\begin{array}{ccccccccc}
& &  &  & & m & &\\
& & | & 0 & 1 & 2 & 3 & 4 & 5 \\  \hline
& 1 & | & 4/3 & 3/2 & - & - & - & - \\
& 2 & | & 9/8 & 8/7 & 5/4 & - & - & - \\
k & 3 & | & 16/15 & 15/14 & 12/11 & 7/6 & - & - \\
& 4 & | & 25/24 & 24/23 & 21/20 & 16/15 & 9/8 & - \\
& 5 & | & 36/35 & 35/34 & 32/31 & 27/26 & 20/19 & 11/10 
\end{array}.
$$

As described in Theorem \ref{thm14}, each entry in this table corresponds with an exact cutoff for $n$--contractivity, where $n=k(k+2)-m^2$; that is, the perturbation is $n$--contractive if and only if $x \leq \frac{n+1}{n}$, so, for example, $(4,3)$--PD corresponds with the shift being $15$--contractive. \  In general there is not an exact correspondence with any $k$--hyponormal condition: for instance, the $(4,2)$--PD cutoff ($21/20$), falls between the cutoffs for $3$-- and $4$--hyponormality ($16/15$ and $25/24$, respectively). \  (As noted above, the cutoff for $k$--hyponormality is exactly that for $(k,0)$--PD.)

It is worth noting that the cutoffs for $m=1$ and $m=0$ are close (the difference amounts to $\ddfrac{1}{k(k+2)(k(k+2)-1)}$); recall that the latter is exactly the threshold for $k$--hyponormality, while, by Theorem \ref{thm114}(ii), the former is the cutoff for the standard matrix for $k$--hyponormality to be CPD. \  Observe also that the cutoff for $(k,2k)$--PD, which is the same as for $2k$--contractivity, is considerably bigger than that for the standard matrix for $k$--hyponormality to be CPD. \  

This shows in particular that, for that matrix to be non-negative as a quadratic form against all the ``negative binomial'' vectors of appropriate size (e.g., \linebreak $(0,1,-1, 0, 0, \ldots), (0,0,1,-3,3,-1,0, \ldots)$ is a weaker property than for it to be CPD. \ This is because if the operator is $2k$--contractive it is, being a contraction, $2k$--hypercontractive, and the quadratic form results just mentioned are all necessarily non-negative. \

It is therefore reasonable to raise the following question.

\begin{question}
Given integers $k\ge 1$ and $m$, with $0 \le m \le k$, what is the set of vectors upon which the original matrix of size $k+1$ is non-negative as a quadratic form if it is $(k,2m)$--PD? \ If $m=0$ it is the whole space; if $m=1$ it is the subspace of dimension $k$ consisting of those vectors whose components sum to zero; if $m=k$ it contains at least the ``negative binomial'' vectors.
\end{question}

\begin{remark}
\rm{
From the above, it is easy to show that the various conditions in $k$ and $m$ are distinct from one another (even just using $j=2$, the Bergman shift). \ It is also true in this setting that there is an easy expression in $x$ equivalent to the shift $A_j(x)$ $p$--contractive:  
$$
x \leq \frac{p+j-1}{p}
$$  
(cf. Theorem \ref{AdamsExner}). \ It results from this that the condition in Theorem \ref{AdamsExner}(i) is satisfied if and only if $A_j(x)$ is $((j+k+m)(k-m)+2m)$--contractive.  This allows us to use contractivity as a ``scale'' to measure and compare the strengths of the various positivities of matrices of certain sizes and certain order of differences.
}
\end{remark}

The following table for $A_2$ illustrates this (we suppress $j=2$ to ease the notation slightly).  Recall that if a weighted shift is $n$--hyponormal then it is $2n$--contractive, and that if we begin with $n$--hyponormality, the process terminates in a $1$ by $1$ matrix whose positivity is exactly this $2n$--th Agler expression.  Further, each of the intermediate conditions is at least no stronger than the previous one.  (In fact, the result above shows that at least in this case the previous condition is strictly stronger than the subsequent one obtained by introducing another $\nabla^{2}$ and decreasing the size of the matrix.) Note also that positive definiteness of an entry in one column is equivalent to conditional positive definiteness of the entry one column to the left. Recall finally that for a contraction, if all the even order contractivity tests are positive, the contraction is subnormal, so each of these ladders has, for contractions, subnormality at the top. 

\vspace*{-.5cm}
$$
\hspace*{-.5cm}
\begin{array}{c|c|c|c|c|c}
&n-\mbox{\rm hyponormal}&&&&\\
n&  (\gamma^{(x)})_{n+1} \geq 0 &  (\nabla^{2} \gamma^{(x)})_{n} \geq 0 & (\nabla^{4} \gamma^{(x)})_{n-1} \geq 0&(\nabla^{6} \gamma^{(x)})_{n-2} \geq 0&(\nabla^{8} \gamma^{(x)})_{n-3} \geq 0 \\ \hline
1& \mbox{3-C} & \mbox{2-C}& -& - & -\\
2& \mbox{8-C}& \mbox{7-C}& \mbox{4-C}& - & - \\
3& \mbox{15-C}& \mbox{14-C}& \mbox{11-C}& \mbox{6-C}& - \\
4& \mbox{24-C}& \mbox{23-C}& \mbox{20-C} & \mbox{15-C}& \mbox{8-C} \\
\end{array}
$$

Put informally, every regularity you see in this display is correct (for example, in any row the ``decreases'' in contractivity have gaps 1, 3, 5, 7, \ldots).  The corresponding results for $A_3$ and $A_4$ are similarly orderly.

We conclude this section with a brief digression, to report another fact about these zeroth weight perturbations $A_j(x)$ of the Agler shifts $A_j$.

\begin{proposition}
Let $A_j(x)$ be the zeroth weight perturbation of the Agler shift $A_j$ with weight sequence $\alpha^{(j)}$ as in Theorem \ref{thm14}. \  Then the weights-squared of $A_j(x)$ are $m$--alternating if and only if
\begin{equation}  \label{eq:naltAjofx}
x \leq 1 + \frac{(j-1) m!}{\prod_{i=1}^{m} (j+i)}.
\end{equation}
If we take $j=2$, that is, the case of the Bergman shift, the weights-squared \linebreak of $A_2(x)$ are $m$--alternating if and only if $A_2(x)$ is $\frac{(m+1)(m+2)}{2}$--contractive.
\end{proposition}

\begin{proof}
As usual with zeroth weight perturbations, the only condition to check for $m$--alternating is that beginning with $\gamma_0(x)$, since we know that $A_j$ has weights-squared that are completely alternating (\cite[Remark 2.9]{BCE}). \  The relevant expression is
\begin{eqnarray*}
\frac{x}{j} - \binom{m}{1} \frac{2}{j+1} & + & \binom{m}{2} \frac{3}{j+2} - \ldots \pm \binom{m}{m} \frac{m+1}{m + j}  \\
&=\!\!&\frac{x}{j} - \frac{1}{j} \!+\! \left(\frac{1}{j} - \binom{m}{1} \frac{2}{j+1} +  \binom{m}{2} \frac{3}{j+2}- \ldots \pm \binom{m}{m} \frac{m+1}{m + j}\right) \\
&=& \frac{x}{j} - \frac{1}{j} + (\nabla^{m}\alpha^{(j)})(0).
\end{eqnarray*}

One may compute using an induction argument that
$$
\nabla^{m}(\alpha^{(j)})(n) = \frac{(1-j)m!}{\prod_{i=n}^{n+m}(j+i)},
$$
and this plus an easy calculation yields the result. \  The assertion for the Bergman shift is simply another calculation using \eqref{eq:naltAjofx}.  \qed
\end{proof}

One may check  that for $A_j(x)$ with $j > 2$ the cutoffs for $m$--alternating do not (necessarily) correspond to some $n$--contractivity;  the $n$--contractivity cutoff is $\frac{n+1}{n}$, and what comes out of the expression above need not correspond to an integer $n$. \  When it does, the correspondence is exact. \ In some sense, the cutoffs correspond approximately with $n$--contractivity where $n = \mathcal{O}(m^2)$, as is to be expected from the formula in \eqref{eq:naltAjofx}.

%%%%%%%%%%%%%%%%%%%%%%%%%%
%%%%%%%%%%%%%%%%%%%%%%%%%%
%%%%%%%%%%%%%%%%%%%%%%%%%%
%%%%%%%%%%%%%%%%%%%%%%%%%%
%%%%%%%%%%%%%%%%%%%%%%%%%%
%%%%%%%%%%%%%%%%%%%%%%%%%%

%%%%%%%%%%%%%%%%%%%%%%%%%%%%%%%
%%%%%%%%%%%%%%%%%%%%%%%%%%%%%%%
%%%%%%%%%%%%%%%%%%%%%%%%%%%%%%%
%%%%%%%%%%%%%%%%%%%%%%%%%%%%%%%
%%%%%%%%%%%%%%%%%%%%%%%%%%%%%%%
%%%%%%%%%%%%%%%%%%%%%%%%%%%%%%%
\bigskip

\noindent \textbf{Acknowledgments.} \ The authors are deeply grateful to the referee for a detailed reading of the paper, and for detecting an inconsistency in the use of the phrase ``infinite divisibility." \ The authors also wish to express their appreciation for support and warm hospitality during various visits (which materially aided this work) to Bucknell University, the University of Iowa, and the Universit\'{e} des Sciences et Technologies de Lille, and particularly the Mathematics Departments of these institutions. \ Several examples in this paper were obtained using calculations with the software tool \textit{Mathematica} \cite{Wol}.
 
\bigskip
\bigskip

%\begin{center}
%\textsc{References}
%\end{center}

\end{document}